\documentclass[12pt]{amsart}
\usepackage{amssymb}

\newtheorem{theorem}{Theorem}[section]
\newtheorem{proposition}[theorem]{Proposition}
\newtheorem{lemma}[theorem]{Lemma}
\newtheorem{corollary}[theorem]{Corollary}
\newtheorem{definition}[theorem]{Definition}
\newtheorem{remark}[theorem]{Remark}

\newcommand{\cE}{{\mathcal E} }
\newcommand{\cF}{{\mathcal F} }
\newcommand{\cG}{{\mathcal G} }

\newcommand{\cO}{{\mathcal O} }

\newcommand{\wt}{\widetilde}

\def\ol#1{{\overline{#1}}}

\def\ka{{K{\"a}h\-ler}}

\def\ii{\sqrt{-1}}

\def\C{\mathbb{C}}

\def\z{\mathfrak z}
\def\cinf{C^\infty}
\def\rk{{\mathrm{rk}}}
\def\tr{{\mathrm{tr}}}

\def\db{{\ol\partial}}

\numberwithin{equation}{section}

\begin{document}

\title{Yang--Mills equation for stable Higgs sheaves}

\author[I. Biswas]{Indranil Biswas}

\address{School of Mathematics, Tata Institute of Fundamental
Research, Homi Bhabha Road, Mumbai 400005, India}

\email{indranil@math.tifr.res.in}

\author[G. Schumacher]{Georg Schumacher}

\address{Fachbereich Mathematik und Informatik,
Philipps-Universit\"at Marburg, Lahnberge, Hans-Meerwein-Strasse,
D-35032 Marburg, Germany}

\email{schumac@mathematik.uni-marburg.de}

\subjclass[2000]{53C07, 32L05}

\keywords{Higgs sheaf, Kobayashi-Hitchin correspondence,
polystability}

\date{}

\begin{abstract}
We establish
a Kobayashi-Hitchin correspondence for the stable Higgs sheaves
on a compact K\"ahler manifold. Using it, we also obtain
a Kobayashi-Hitchin correspondence for the stable
Higgs $G$--sheaves,
where $G$ is any complex reductive linear algebraic group.
\end{abstract}
\maketitle

\section{Introduction}

The concept of Hermite--Einstein equations for stable sheaves was
introduced by Bando and Siu in \cite{BS}. It depends upon the notion
of a certain class of hermitian metrics on reflexive sheaves,
called \textit{admissible}, for which the curvature is square
integrable satisfying a pointwise
boundedness condition. It also
depends upon the extension of a solution of the corresponding
Yang--Mills (= Hermite--Einstein)
equation to the open subset where the sheaf is locally
free. The approach is based on solving a heat equation.

The notion of a \textit{Higgs bundle} is due to Hitchin and
Simpson. They generalized the definition of stability to Higgs
bundles, and also generalized the Yang--Mills equation to the
Higgs bundles. In \cite{Hit,Si}
they established for Higgs bundles what is called the
Kobayashi--Hitchin correspondence.

Our aim here is to combine both concepts to get a
Kobayashi--Hitchin correspondence for Higgs sheaves. This is
worked out in Theorem \ref{yang-mills-higgs}.

Also, stable principal Higgs $G$--sheaves are introduced,
where $G$ is any complex reductive linear algebraic group. It
follows from our main theorem that tensor products of polystable
Higgs sheaves are again polystable. Once shown this fact, a
Kobayashi--Hitchin correspondence for
stable Higgs $G$--sheaves follows.


\section{Higgs sheaves and admissible metrics}

Let $X$ be a compact connected K\"ahler manifold equipped with a
K\"ahler form $\omega$. The adjoint of multiplication of differential
forms by $\omega$ will be denoted by $\Lambda_\omega$.
We will use the summation convention throughout.

\begin{definition}
  {\rm A \textit{Higgs sheaf} on $X$ consists of a
  torsionfree sheaf $\cE$ on $X$ together with a holomorphic section
  $\varphi=\varphi_\alpha dz^\alpha$ of
  $\Omega^1_X(\textit{End}(\cE))$ such that
  the form $\varphi\wedge\varphi=[\varphi_a,\varphi_\gamma]dz^\alpha
  \wedge dz^\gamma$, which is a
  holomorphic section of $\Omega_X^2(\textit{End}( \cE))$,
  vanishes identically.}
\end{definition}

Let $\cE$ be a torsionfree sheaf on $(X,\omega)$.
Let $S \subset X$ be the locus where $\cE$ is not locally free. So
$S$ is a complex analytic subset with $\mathrm{codim}_XS\geq 2$.
Following \cite{BS} we call a hermitian metric $h$ on
the holomorphic vector bundle $\cE|X\backslash S$ to
be \textit{admissible} if the following two hold:
\begin{itemize}
\item[$\mathbf{A_1}$:] \quad the curvature tensor
$F$ of $h$ is square integrable, and
\item[$\mathbf{A_2}$:] \quad
$\Lambda_\omega F$ is bounded.
\end{itemize}

For a torsionfree sheaf $\cE$, there is a natural embedding
of $\cE$ in its double dual $\cE^{\vee\vee}$. It is easy to
see that any 
admissible hermitian metric on $\cE^{\vee\vee}$ restricts
to an admissible hermitian metric on $\cE$.

Let $n$ be the complex dimension of $X$. The degree of
a torsionfree sheaf $\cE$, which is same as the degree of the
determinant line bundle $\det \cE$ \cite[Ch. V, \S 6]{Ko}, is
defined in terms of cohomology classes:
$$
\deg \cE = (c_1(\cE) \cup [\omega]^{n-1})\cap[X]\in {\mathbb R}\, ,
$$
where $[\omega]\in H^2(X, {\mathbb R})$ is the cohomology
class represented by $\omega$. Here, and below, we denote by
$\omega^k$ the $k$--the exterior power of $\omega$ divided by $k!$.

Let $(\cE,\varphi)$ be a Higgs sheaf. A coherent subsheaf
$\cF\subset \cE$ is called a \textit{Higgs subsheaf}, if
$\varphi(\cF)\subset \Omega^1_X\otimes\cF$.

\begin{definition}\label{de.st.sh.}
  {\rm A Higgs sheaf $(\cE,\varphi)$ over $X$
is called \textit{stable} (respectively, \textit{semistable}), if
for any
  Higgs subsheaf $\cF$, with $0<\rk \cF < \rk \cE$, the inequality
$$
\frac{\deg \cF}{\rk \cF}< \frac{\deg \cE}{\rk \cE}
$$
(respectively, $\frac{\deg \cF}{\rk \cF}\leq
\frac{\deg \cE}{\rk \cE}$)
holds. A Higgs sheaf $(\cE,\varphi)$
is called \textit{polystable}, if it decomposes into a
direct sum of stable Higgs subsheaves
$(\cE,\varphi) = \oplus_i (\cF_i, \varphi_i)$ with
$$
\frac{\deg \cF_i}{\rk \cF_i}= \frac{\deg \cE}{\rk \cE}
$$
for all $i$.}
\end{definition}

\begin{lemma}
The degree of a torsionfree sheaf $\cE$ on $X$ can be
evaluated from the curvature of an admissible hermitian
metric $h$ on $\cE$. Namely,
\begin{equation}
\deg \cE = \int_X c_1(\cE, h)\wedge \omega^{n-1}\, .
\end{equation}
\end{lemma}

\begin{proof}
Let $\eta$ denote the (singular) hermitian metric on the
determinant line bundle $\det\cE$ which is
induced by the given admissible hermitian metric $h$ on $\cE$. Now
let $\eta_0$ be a background metric on $\det\cE$ of class $\cinf$,
and set $\eta=e^\chi\cdot \eta_0$, where $\chi$ is a real valued
function on the open subset $X\backslash S$ where $\cE$ is locally
free. Since the hermitian metric on $\cE$ is
admissible, we know that $\chi$ is of class $\cinf$ on $X\backslash
S$, and $\Box \chi$ is bounded. (Note that the condition that
$\Lambda_\omega F$ is bounded implies that $\Box \chi$ is bounded.)

The degree of $\cE$ with respect to $\omega$ can be
computed from $\eta_0$.
Therefore, to prove the lemma it suffices to show that
\begin{equation}
\label{inlapl0} \int_X \Box \chi \cdot \omega^n =0\, ,
\end{equation}
where $\chi$ is of class $\cinf$ on $X\backslash S$ with $\Box\chi$
bounded. Since $f:=\Box \chi$ is in $L_2$, we can apply the
global Green's operator to the function $f_0= f -\alpha$, with
$\alpha:=(\int_Xf \omega^n/\int_X\omega^n)$, and obtain a solution
$\rho \in \mathbf{H}^2$ of $\Box \rho = f_0$ (cf.\ \cite[Section
7]{kodaira}). By elliptic regularity, the function $\rho|X\backslash
S$ on $X\backslash S$ is of class $\cinf$. Let $\{U_i\}$ be an open
covering of $X$ together with a set of $\cinf$ functions $\mu_{U_i}$
on $U_i$ satisfying $\Box \mu_{U_i}=\alpha$ on $U_i$. Then
$(\chi-\rho-\mu_{U_i})|U_i\backslash S$ is harmonic.
We note that the function $(\chi-\rho-\mu_{U_i})|U_i\backslash S$
can be
extended as a harmonic function across the set $S\cap U_i$. Indeed,
this follows from the fact that the complex codimension of
the analytic subset $S\cap U_i\subset
U_i$ is at least two. In particular,
$\chi \in \mathbf{H}^2$, and \eqref{inlapl0} follows.
\end{proof}

Let $(\cE,\varphi)$ be a Higgs sheaf. Then $\varphi$ extends
to a Higgs field on the double dual $\cE^{\vee\vee}$. This
Higgs field on $\cE^{\vee\vee}$ will be denoted by
$\varphi^{\vee\vee}$.

\begin{lemma}\label{tors}
A torsionfree Higgs sheaf $(\cE,\varphi)$ on $(X,\omega)$ is
stable (respectively, semistable),
if and only if $(\cE^{\vee\vee},\varphi^{\vee\vee})$ is
stable (respectively, semistable). Similarly, $(\cE,\varphi)$
is polystable if and only if $(\cE^{\vee\vee},\varphi^{\vee\vee})$
is so.
\end{lemma}

\begin{proof}
We consider the embedding $\cE\hookrightarrow\cE^{\vee\vee}$,
and note that the degrees of $\cE$ and $\cE^{\vee\vee}$ coincide.
If $\cF$ denotes a Higgs
subsheaf of $(\cE,\varphi)$, then its image in
$\cE^{\vee\vee}$ is a Higgs subsheaf of same
rank and degree as those of $\cF$. So the
stability of $(\cE^{\vee\vee},\varphi^{\vee\vee})$ implies the
stability of $(\cE,\varphi)$. Conversely, for any Higgs subsheaf
$\cG \subset \cE^{\vee\vee}$, the intersection $\cG\cap \cE$ is a
Higgs subsheaf of $\cE$ same rank and degree. This proves the
converse. A similar proof works in the semistable and
polystable cases.
\end{proof}

We will define a hermitian Yang--Mills--Higgs metric
on a Higgs sheaf.

\begin{definition}\label{de.h-y-m}
{\rm Let $(\cE,\varphi)$ be a Higgs sheaf on $(X,\omega)$. An
admissible hermitian metric $h$ on $\cE$ is called a \textit{hermitian
Yang--Mills--Higgs metric} if the hermitian
Yang--Mills--Higgs equation
\begin{equation}\label{h-y-m}
\Lambda_\omega (F+[\varphi,\varphi^*])
= \lambda \cdot \textrm{Id}_\cE
\end{equation}
on $X\backslash S$ is satisfied for some constant $\lambda
\in {\mathbb C}$; here $S$, as before, is the subset
of $X$ where $\cE$ fails to be locally free.}
\end{definition}

We make the following observation: Let $U\subset \C^n$ be an open
subset, and let $\cE_U \subset \cO^\ell_U$ be a coherent subsheaf.
Then the restriction of any hermitian metric on $\cO^\ell_U$ to
$\cE_U$ is an admissible metric over relatively compact subsets of
$U$.

Let $\cE \in Coh(X)$ be torsionfree. There exists an open covering
$\{U_\alpha\}$ of $X$ together with presentations
\begin{equation}\label{resol}
\cO^{k_\alpha}_{U_\alpha}\overset{\varphi^\vee_\alpha}{\longrightarrow}
\cO^{\ell_\alpha}_{U_\alpha} \longrightarrow \cE^\vee|U_\alpha
\longrightarrow 0\, .
\end{equation}
According to Hironaka's flattening theorem \cite{hiro_flat}, there
exists a finite sequence of blow--ups with smooth centers
$\pi_i:X_i \to X_{i-1}$, where $i=1,\ldots,\ell$ and
$X_0=X$, such that the
pull--back of $\cE^\vee$ to $X_\ell$ modulo torsion is
locally free. Set $\wt X :=X_\ell$, and let
$$
\pi:=\pi_1\circ\ldots\circ \pi_\ell: \wt X \to X
$$
be the projection. So $(\pi^*(\cE^\vee))/\text{torsion}$
is a holomorphic vector bundle over $\wt X$.

Set $\wt U_\alpha:=\pi^{-1}(U_\alpha)$.
We note that for a resolution of type
\eqref{resol}, the equality
$$
\textrm{coker}(\pi^*\varphi^\vee_\alpha)=(\pi^*(\cE^\vee)/
\text{torsion})|\wt U_\alpha
$$
holds independent of the choice of a local resolution
$\pi\vert\wt U_\alpha : \wt U_\alpha\rightarrow U_\alpha$
(cf.\ \cite[Section 5]{rey-gr} in particular \cite[5.4]{rey-gr}).

Assume that $\cE$ is reflexive, that is,
$\cE=\cE^{\vee\vee}$. We have exact sequences
$$
0 \longrightarrow \cE|\wt U_\alpha \longrightarrow\cO_{\wt
U_\alpha}^{ \ell_\alpha}\overset{\varphi_\alpha}{ \longrightarrow}
\cO_{\wt U_\alpha}^{k_\alpha}\, .
$$
Setting $\wt\cE:= (\pi^*(\cE^\vee)/\text{torsion})^\vee$
we get that
$$
\wt\cE|\wt U_\alpha=\ker(\pi^*\varphi_\alpha)=
(\pi^*\cE/\text{torsion})|\wt U_\alpha
$$
is actually a \textit{sub--vector bundle} of $\cO_{\wt
U_\alpha}^{\ell_\alpha}$, where
$\wt U_\alpha=\pi^{-1}(U_\alpha)$.
Observe that $\cE|U_\alpha \hookrightarrow
\cO^{\ell_\alpha}_{U_\alpha}$ is a subbundle wherever it is
locally free.

Let $h_\alpha$ be a hermitian metric on
$\cO^{\ell_\alpha}_{U_\alpha}$. Let $\wt h_\alpha$ denote the
restriction of the hermitian metric
$\pi^*h_\alpha$ to the subbundle
$\wt \cE|U_\alpha$. Now $\wt h_\alpha$
clearly is a hermitian metric (with no degeneracies) which
descends to a hermitian metric on $\cE|U_\alpha$ over $U_\alpha$
wherever $\cE$ is locally free. Taking a partition of unity
subordinate to $\{U_\alpha\}$, we get a hermitian metric $\wt h$ on
$\wt \cE$, which descends to a hermitian metric on $\cE$ over
$X\setminus S$ (the subset where $\cE$ is locally free). Let
$h$ denote the hermitian metric on $\cE|X\setminus S$ obtained
this way from $\wt h$.

The complex manifold $\wt X$ is \ka. We denote by
$\eta$ a \ka\ form on $\wt X$, and we set
$$
\omega_\epsilon:= \pi^*\omega + \epsilon\eta
$$
for $0 \leq \epsilon\leq 1$.
We will later need the following argument which shows that $h$ is
admissible with respect to $\omega$ on $X$:

Let $F$ be the curvature tensor of $h$ (over $X\backslash S$). The
pull--back of $F$
to $\wt X \setminus \pi^{-1}(S)$ extends to $\wt X$ as the
curvature tensor of the hermitian metric
$\wt h$ on the vector bundle $\wt \cE$. We use
the same notation $F$ for the pull--back.

Now
\begin{equation}\label{LambdaF}
\Lambda_\epsilon F \omega_\epsilon^n= F \wedge
\omega_\epsilon^{n-1}\, .
\end{equation}
The right--hand side is a bounded $\textit{End}(\wt\cE)$--valued form
on $\wt X$ with estimates uniform with respect to $\epsilon$. With
$\epsilon\to 0$, since $\omega_0=\pi^*\omega$, we see the boundedness
of $\Lambda_\omega F$
on $X\backslash S$. In a similar way we see that
$$
\tr (F\wedge F)\wedge \omega^{n-2}= \lim_{\epsilon\to 0} \tr(F\wedge
F)\wedge \omega_\epsilon^{n-2}
$$
is integrable on $X$. In view of \eqref{LambdaF}, it therefore
follows that $F$ is square integrable. Consequently, $h$ is
admissible with respect to $\omega$.

\begin{lemma}
Let $(\cE,\varphi)$ be a Higgs sheaf equipped with
an admissible hermitian metric $h$ on the
compact \ka\ manifold  $(X,\omega)$.
Then $\varphi$ is bounded on $X$. In particular,
$\varphi$ is square integrable.
\end{lemma}

\begin{proof}
We consider $\tr (\varphi\wedge\varphi^*)\wedge \omega^{n-1}$ pulled
back to $\wt X$, where it is easily seen to be bounded in terms of
$\omega^n$.
\end{proof}

\section{Heat equation for Higgs sheaves}

The following is the main result proved here.

\begin{theorem}\label{yang-mills-higgs}
Let $(X,\omega)$ be a compact \ka\ manifold and $(\cE,\varphi)$ a
stable (torsionfree) Higgs sheaf on $(X,\omega)$. Then there exists
an admissible hermitian Yang--Mills--Higgs metric on $(\cE,\varphi)$.
Furthermore, the admissible hermitian Yang--Mills--Higgs
connection is unique.
\end{theorem}

This theorem will be proved later.

\begin{lemma}
It suffices to prove the theorem under the assumption
that $\cE$ is reflexive.
\end{lemma}

\begin{proof}
By Lemma~\ref{tors}, the Higgs sheaf
$(\cE^{\vee\vee},\varphi^{\vee\vee})$ is stable. A Hermitian
Yang--Mills--Higgs metric on $\cE^{\vee\vee}$ can be restricted to
$\cE$ as such, since the equation
\eqref{h-y-m} in Definition \ref{de.h-y-m}
is required only on the locally free locus of $\cE$.

We mention that for the existence proof it would be sufficient to
relax the condition of admissibility of the initial hermitian
metric to require the conditions $\mathbf{A_1}$ and $\mathbf{A_2}$
to hold on the complement of some complex
analytic subset in $X$, of codimension
at least two, that contains the set $S$ where $\cE$ fails to be
locally free. In this sense an admissible
metric for $\cE$ would also yield an admissible metric for
$\cE^{\vee\vee}$. It would follow from \cite[Theorem 2c]{BS} that
\eqref{h-y-m} holds on the locally free locus of the given sheaf.

For the uniqueness of the connection, we observe that any admissible
hermitian metric on $\cE$ can be interpreted as an admissible metric
on the double dual $\cE^{\vee\vee}$ in the above slightly more
general sense. A posteriori it satisfies \eqref{h-y-m} wherever
$\cE^{\vee\vee}$ is locally free.
\end{proof}

We first assume that $\cE$ is already a vector bundle on $X$,
equipped with an hermitian metric $h$. We consider the ``augmented
curvature''
\begin{equation}\label{Ftilda}
\wt F = \varphi\wedge\varphi^* + \varphi^*\wedge\varphi + F\, ,
\end{equation}
where $F$ is the curvature of $h$.
For a differentiable family $h_t$ of hermitian metrics $h_t$, $t\geq
0$, we denote the curvature by $F_t$ and set
$$
\wt F_t = \varphi\wedge\varphi^* + \varphi^*\wedge\varphi + F_t\, ,
$$
where the adjoint forms $\varphi^*$ are taken with respect to $h_t$.
The heat equation in the sense of Higgs bundles is
\begin{equation}\label{heateq}
\frac{dh_t}{dt}\cdot h_t^{-1}=-(\ii \Lambda_\omega \wt F_t - \lambda
\cdot \textrm{Id}_\cE)\, ,
\end{equation}
with initial metric $h_0=h$; the constant $\lambda$ is determined by
$$
\int_X\tr (\ii\Lambda_\omega\wt F_t- \lambda\cdot
\textrm{Id}_\cE)\wedge \omega^n
=0\, .
$$
In the latter equation $\wt F_t$ can be replaced by $F_t$, as
$\tr(\varphi\wedge \varphi^* + \varphi^*\wedge \varphi)=0$.

The quantity $\wt F$ in \eqref{Ftilda} and the corresponding heat
equation \eqref{heateq} can be interpreted in terms of a certain
connection, induced by both the hermitian metric and the Higgs
field. However, we will not take this standpoint in our arguments.

The standard equations and estimates for Higgs bundles are formally
the same as in the classical case.

\begin{lemma}
Let $\Box$ denote the Laplacian for differentiable sections of
$\textit{\textit{End}}(\cE)$ on $(X,\omega)$. Then
  \begin{equation}\label{dotcurv}
 \frac{d}{dt}\Lambda_\omega\wt F_t= \Box \Lambda_\omega\wt F_t\, .
 \end{equation}
\end{lemma}

\begin{proof}
It follows from \eqref{heateq} that
\begin{eqnarray*}
 \frac{d}{dt}\Lambda_\omega\wt F_t &=&
\frac{d}{dt}\Lambda_\omega(\db(\partial
  h_t \cdot h_t^{-1}))\\
  &=& \Lambda_\omega\db\left(\partial(\frac{d}{dt}h_t
 \cdot h^{-1}_t) +[\frac{d}{dt} h_t
  \cdot h^{-1}_t,\partial h_t\cdot h_t^{-1}]\right)\\
  &=& \Lambda_\omega\db \partial_{h_t}(\frac{d}{dt} h_t
  \cdot h^{-1}_t
  )\\
  &=& \Box \Lambda_\omega\wt F_t\, ,
\end{eqnarray*}
where the connection $\partial_{h_t}$ is taken with respect to
$h_t$.
\end{proof}
Next, equation \eqref{dotcurv} implies immediately that
\begin{equation}
  \frac{d}{dt}(|\Lambda_\omega\wt F_t|^2)=\Box(|\Lambda_\omega\wt
  F_t|^2)-|\nabla\Lambda\wt F_t|^2\, .
\end{equation}
Also, equation \eqref{dotcurv} yields an estimate of
differentiable functions
\begin{equation}\label{estheat}
  \frac{d}{dt}|\Lambda_\omega\wt F_t| \leq \Box |\Lambda_\omega
\wt F_t|\, .
\end{equation}
{}From there
\begin{equation}\label{intLambda}
  \frac{d}{dt} \int_X |\Lambda_\omega\wt F_t|^2 \omega^n = - \int_X
|\nabla \wt F_t|^2 \omega^n
\end{equation}
and
\begin{equation}
\int_X |\Lambda_\omega \wt F_t| \omega^n \leq \int_X |\Lambda_\omega
\wt F_0|\omega^n\, .
\end{equation}
Now the estimate \eqref{estheat} implies
\begin{equation}
  |\Lambda_\omega\wt F_t|(x) \leq \int_X H(t,x,y)|\Lambda_\omega\wt
F_0(y)|
  \omega(y)^n,
\end{equation}
where $H(t,x,y)$ is the heat kernel for differentiable functions on
$(X,\omega)$.

The finite time solutions of the heat equation \eqref{heateq} are
guaranteed by a result of Simpson \cite{Si}, as well as the
convergence of a subsequence of the hermitian metrics to a solution
of the hermitian Yang--Mills--Higgs equation \eqref{h-y-m} after
applying suitable gauge transformations.

Now following \cite{BS} we consider the heat equation \eqref{heateq}
for $(\wt\cE,\pi^* \varphi,\wt h)$ on $(X,\omega_\epsilon)$ with $0<
\epsilon\leq 1$. No assumption of stability is needed in order to
get solutions for all finite $t$ according to \cite{Si}.

For $\wt \cE$ on $(\wt X,\omega_\epsilon)$ we consider the heat
equation \eqref{heateq} and denote the solutions by $\wt
h_{t,\epsilon}$, with augmented curvatures $\wt F_{t,\epsilon}$,
and as before we set $\Lambda_\epsilon$ to be the adjoint
of exterior multiplication
with the form $\omega_\epsilon$. Like in \cite{BS}
the equalities \eqref{dotcurv} --- \eqref{estheat} on $(\wt X,
\omega_\epsilon)$ for $(\wt\cE,\wt h)$ together with
\cite[Proposition 2]{BS} imply that
$$
\Lambda_\epsilon\wt F_{t,\epsilon}\in C^\infty({\wt X}, End(\wt \cE))
$$
are uniformly bounded with respect to $0<\epsilon\leq 1$ and $t\geq
0$. Next, \cite[Lemma 6]{BS} together with the boundedness of
$\Lambda_\epsilon\wt F_{t,\epsilon}$ implies that $\wt
F_{t,\epsilon}$ is square integrable with uniform bound on the norm.
As an application we note that the solutions $h_{t,\epsilon}$ of
\eqref{heateq} are ``uniformly admissible''. The limits
$$
h_t = \lim_{\epsilon\to 0} h_{t,\epsilon}
$$
solve the heat equation on $X\backslash S$ for $\cE$ with admissible
$h_t$ (with curvatures $F_t$) for all $t\geq 0$ and uniform bounds
for $|\Lambda F_t|$ and $\|F_t\|_{L^2}$.

Now with $\epsilon\to 0$ the equation  \eqref{intLambda} holds on
$(X,\omega)$ for $(\cE, h_t)$ implying that
$$
\int_0^\infty \int_X |\nabla\Lambda_\omega \wt F_t|^2 \omega^n \leq
\int_X |\Lambda_\omega\wt F_0|^2 \omega^n\, .
$$
In particular, there exists a sequence of real
numbers $t_i \to \infty$ such that
$$
\int_X |\nabla\Lambda_\omega \wt F_{t_i}|^2 \omega^n \to 0\, .
$$
Under the assumption of stability of
$(\cE,\varphi,\omega)$, on the complement (in $X\backslash S$)
of a subset $S'\subset
X\backslash S$ of finite Hausdorff measure in real codimension four
there exists a subsequence $h_{t_{i(j)}}$ which converges
to a limit $h_\infty$
(after applying suitable gauge transformations). The limit is a
hermitian Yang--Mills--Higgs metric on this part \cite[p. 895]{Si}.
Let $\wt F$ be the augmented curvature of the limit metric. Now
$\Lambda_\omega\wt F$ is bounded, in particular
$\Lambda_\omega F$ is
bounded. By \cite[Theorem 2]{BS} the hermitian metric $h_\infty$ is
in $L^p_{2\textit{loc}}(X\backslash S)$, also $h_\infty$ is locally
bounded on $X\backslash S$. Finally, the ellipticity of the
hermitian Yang--Mills--Higgs equation implies the regularity of
$h_\infty$ on all of $X\backslash S$.

This proves the existence part of Theorem~\ref{yang-mills-higgs}.

The uniqueness part needs some further preparation.

Considering the adjoint action of the Higgs field $\varphi$ on
$\cF:=\textit{End}(\cE)=\cE\otimes \cE^\vee$, we obtain a Higgs
field $\wt\varphi$ on $\textit{End}(\cE)$. Let $h$ and $h'$ be two
admissible hermitian Yang--Mills--Higgs metrics on $(\cE,\varphi)$
with connections $\theta$ and $\theta'$. Then it follows immediately
that $h'^{-1}$ is admissible on $(\cE^\vee, \varphi^\vee)$ and
it is a
hermitian Yang--Mills--Higgs metric. Now
$$
\theta_\cF:=\theta \otimes
id_{\cE^\vee} - id_\cE \otimes \theta'^\vee
$$
is a hermitian Yang--Mills--Higgs
connection on the sheaf of endomorphisms $(\cF, \wt\varphi)$ induced
by $h_\cF=h \otimes h'^{-1}$.

\begin{lemma}\label{flat}
Let $\sigma \in H^0(X,\cF)$ be
any holomorphic section which commutes with the Higgs field, i.e.,
$[\sigma,\varphi]=0$. Then $\sigma$ is parallel with respect to
$\theta_\cF$ over the locally free locus of $\cF$.
\end{lemma}

\begin{proof}
On the complement $X\backslash S$ of the singular locus of the
sheaf $\cF$, the pointwise norm of any given section $\sigma$
satisfies
$$
\Box(|\sigma|^2)(x) = |\partial_{\theta_\cF} \sigma|^2 (x)- \langle
[\Lambda_\omega  F,\sigma],\sigma\rangle(x)\, .
$$
As $\deg \cF=0$, and $[\sigma,\varphi]=0$, we have (pointwise)
$$
- \langle [\Lambda_\omega
F,\sigma],\sigma\rangle=\langle\Lambda_\omega[\wt\varphi,
\wt\varphi^*](\sigma),\sigma^*
\rangle=\langle\Lambda_\omega[\varphi,
[\varphi^*,\sigma]],\sigma\rangle=\langle
[\varphi^*,\sigma],[\varphi^*,\sigma]\rangle\geq 0\, .
$$
Hence
\begin{equation}\label{laplgeqz}
\Box(|\sigma|^2) \geq |\partial_{\theta_\cF}\sigma|^2 \geq 0
\textrm{\quad over \quad } X\backslash S\, .
\end{equation}
Now by \cite[Theorem 2]{BS} the pointwise norm $|\sigma|$ is
bounded, and $|\sigma|^2$ can be extended to $X$ as a subharmonic
function. The maximum principle shows that $|\sigma|$ is constant so
that \eqref{laplgeqz} implies $\partial_{\theta_\cF} \sigma=0$.
Hence $\sigma$ is a flat section.
\end{proof}

Now we prove the uniqueness part of Theorem~\ref{yang-mills-higgs}.

The assumptions of Lemma~\ref{flat} are satisfied for
$\sigma= id_\cE$. So the identity map of
$\cE$ is a flat section of
$\cF$ with respect to $\theta_\cF$, i.e.,
$$
\theta\circ id_\cE= id_\cE \circ \theta'\, ,
$$
and the two connections agree. \qed

\begin{corollary}\label{YM-cor.}
Let $(X,\omega)$ be a compact \ka\ manifold and $(\cE,\varphi)$ a
torsionfree Higgs sheaf on $(X,\omega)$. There exists
an admissible hermitian Yang--Mills--Higgs metric on $(\cE,\varphi)$
if and only if $(\cE,\varphi)$ is polystable.
Furthermore, a polystable Higgs sheaf admits a unique
admissible hermitian Yang--Mills--Higgs connection.
\end{corollary}

\begin{proof}
Since a polystable Higgs sheaf is a direct sum of stable Higgs
sheaves of same slope (= $\frac{\text{degree}}{\text{rank}}$),
it follows from Theorem
\ref{yang-mills-higgs} that a polystable Higgs sheaf
admits a unique admissible hermitian Yang--Mills--Higgs connection.
The decomposition of a polystable Higgs sheaf into a direct sum of
stable Higgs sheaves is an orthogonal decomposition.

To prove the converse, assume that $(\cE,\varphi)$ has an
admissible hermitian Yang--Mills--Higgs metric $h$. So $h$
is a nonsingular hermitian metric on the vector bundle
$\cE\vert{X\setminus S}$ satisfying the
hermitian Yang--Mills--Higgs equation, where $X\setminus S$
is the subset where $\cE$ is locally free. Since the complex
codimension of $S$ is at least two, from the
hermitian Yang--Mills--Higgs equation it follows that
$(\cE,\varphi)$ is polystable (see
\cite[p. 878, Proposition 3.3]{Si}).
\end{proof}

\section{Hermitian Yang--Mills--Higgs connection on
a Higgs $G$--sheaf}

As before, let $X$ be a compact connected K\"ahler manifold
equipped with a K\"ahler form $\omega$.

\begin{definition}
{\rm By a \textit{large open subset of $X$} we will mean a dense
open subset $U$ of $X$ such that the complement $X\setminus
U$ is a complex
analytic subspace of $X$ of complex codimension at least two.}
\end{definition}

Let $G$ be a connected
reductive linear algebraic group defined over
the field of complex numbers. The Lie algebra of $G$ will be
denoted by $\mathfrak g$. Let
\begin{equation}\label{eq0}
{\mathfrak g}'\, := \, [{\mathfrak g}\, , {\mathfrak g}]
\end{equation}
be the semisimple part of $\mathfrak g$. Let $\z({\mathfrak g})
\subset\, {\mathfrak g}$ be the center of ${\mathfrak g}$. The
projection
\begin{equation}\label{0eq0}
{\mathfrak g}\,\longrightarrow\, {\mathfrak g}/\z({\mathfrak g})
\end{equation}
identifies ${\mathfrak g}'$ with the quotient ${\mathfrak
g}/\z({\mathfrak g})$.

We recall from \cite{GS} the definition of a principal $G$--sheaf.

A \textit{principal $G$--sheaf} over $X$ is a triple of the form
$(E_G\, , E \, ,\psi)$, where
\begin{enumerate}
\item $E_G$ is a rational principal $G$--bundle over $X$,
which means that $E_G$ is a holomorphic principal $G$--bundle over
some large open subset $U$ of $X$,

\item for any character $\chi$ of $G$, the holomorphic line
bundle $E_G\times^\chi {\mathbb C}$ over $U$ associated to
$E_G$ for $\chi$ extends to a holomorphic line bundle over $X$,

\item $E$ is a torsionfree coherent analytic sheaf on $X$,
and

\item
\begin{equation}\label{00}
\psi\, :\, E_G({\mathfrak g}') \,\longrightarrow
\, E\vert_U
\end{equation}
is a holomorphic isomorphism of vector bundles over a large open
subset $U$ over which $E_G$ is a holomorphic principal $G$--bundle
and $E$ is locally free,
where $E_G({\mathfrak g}')$ is the vector bundle over $U$ associated
to $E_G$ for the $G$--module ${\mathfrak g}'$ defined in \eqref{eq0}.
\end{enumerate}

We note that the second condition that $E_G\times^\chi {\mathbb C}$
extends to a holomorphic line bundle over $X$ is automatically
satisfied when $X$ is a complex projective manifold. Since the
adjoint bundle $\text{ad}(E_G)$ is the direct sum of $E_G({\mathfrak 
g}')$ with the trivial vector bundle with fiber $\z({\mathfrak g})$,
the fourth condition ensures that the adjoint bundle $\text{ad}(E_G)$ 
over $U$ extends to $X$ as a torsionfree coherent analytic sheaf.

\begin{remark}\label{rem1}
{\rm The above mentioned large open subset $U$ is not a part of the
definition of a principal $G$--sheaf. In other words, we do not
distinguish between the two $G$--sheaves given by $(E\, ,U\, ,E_G\,
,\psi)$ and $(E\, ,U'\, ,E'_G\, ,\psi')$ respectively where
$E_G\vert_{U\cap U'}\, =\, E'_G\vert_{U\cap U'}$ and
$\psi\vert_{U\cap U'} \, =\, \psi'\vert_{U\cap U'}$. However, we may
take $U$ to be the open subset of $X$ over which the
torsionfree coherent analytic sheaf $E$ is a vector bundle. In this
sense, there is a natural choice of the large open subset $U$. We
note that the open subset over which $E$ is a vector bundle
is also the largest open subset over which $E_G$ is a holomorphic
principal $G$--bundle. (See \cite{GS} for the details.)}
\end{remark}

We will now define a principal Higgs $G$--sheaf.

\begin{definition}\label{g-higgs-def}
{\rm A \textit{principal Higgs $G$--sheaf} on $X$ consists
of data of the following type:
\begin{itemize}
\item{} a principal $G$--sheaf $(E_G\, , E \, ,\psi)$ on $X$,

\item{} a holomorphic section
\[
\varphi\, \in\, H^0(U, \, \Omega^1_X\otimes {\rm ad}(E_G))
\]
of the adjoint bundle ${\rm ad}(E_G) \, :=\, E_G({\mathfrak g})$
defined on the large open subset $U\,\subset\, X$ over which $E$ is
locally free (see Remark \ref{rem1}), and

\item{} a holomorphic homomorphism of coherent analytic sheaves
\[
\widehat{\varphi}\, :\, E \, \longrightarrow\, \Omega^1_X\otimes E
\]
\end{itemize}
satisfying the following two conditions:
\begin{enumerate}
\item{} the composition
\[
E \, \stackrel{\widehat{\varphi}}{\longrightarrow}\, \Omega^1_X
\otimes E\, \stackrel{\widehat{\varphi}}{\longrightarrow}\,
\Omega^2_X\otimes E
\]
vanishes identically, and

\item{} the restriction of $\widehat{\varphi}$ to the large open
subset $U$ coincides, using $\psi$ (in \eqref{00}), with the
homomorphism $E_G({\mathfrak g}')\longrightarrow \Omega^1_X
\otimes E_G({\mathfrak g}')$ defined by $\alpha \longmapsto
[\varphi ,\alpha]$.
\end{enumerate}}
\end{definition}

We note that for any principal Higgs $G$--sheaf $(E_G\, , E \,
,\psi\, , \varphi\, ,\widehat{\varphi})$, the pair
$(E \, , \widehat{\varphi})$ is a Higgs sheaf. Similarly, the
pair $(E_G({\mathfrak g}')\, , \varphi)$ is so, and furthermore,
these two Higgs sheaves
are identified using the isomorphism $\psi$.

A principal Higgs $G$--sheaf $(E_G\, , E \, ,\psi\, ,
\varphi\, ,\widehat{\varphi})$ is called
\textit{stable} (respectively, \textit{semistable}) if for
every triple of the form $(U'\, , Q\, , E_Q)$, where
\begin{itemize}
\item $U'\,\subset\, M$ is a large open subset
contained in the open subset of $M$ over which
$E_G$ is a holomorphic principal $G$--bundle,

\item $Q\, \subset\, G$ is a proper maximal parabolic
subgroup, and

\item
\begin{equation}\label{01}
E_Q \,\subset\, E_G\vert_{U'}
\end{equation}
is a holomorphic reduction of structure group of $E_G\vert_{U'}$
to $Q$ over $U'$ such that $\varphi\vert_{U'}$ is a section of
$\Omega^1_X\otimes{\rm ad}(E_Q)$,
\end{itemize}
the following inequality
\begin{equation}\label{0}
\text{degree}({\rm ad}(E_G\vert_{U'})/{\rm ad}(E_Q)) \, > \, 0
\end{equation}
(respectively, $\text{degree}({\rm ad}(E_G\vert_{U'})/{\rm ad}(E_Q))
\,\geq \, 0$) holds.

Since ${\rm ad}(E_Q)$ is a subbundle of ${\rm ad}(E_G)$ over a
large open subset, and ${\rm ad}(E_G)$ extends to $X$ as a
coherent analytic sheaf, it follows that ${\rm ad}(E_Q)$ also
extends to $X$ as a coherent analytic sheaf.

By a \textit{Levi subgroup} of a parabolic subgroup
$P\,\subset\, G$ we will mean a connected reductive
subgroup of $P$ whose projection to the quotient $P/R_u(P)$
is an isomorphism, where $R_u(P)$ is the unipotent radical
of $P$.

A principal Higgs $G$--sheaf
$$
(E_G\, , E \, ,\psi\, , \varphi\, ,\widehat{\varphi})
$$
is called \textit{polystable} if either
$(E_G\, , E \, ,\psi\, ,
\varphi\, ,\widehat{\varphi})$ is stable, or there
is a pair $(L(P)\, , E_{L(P)})$, where
\begin{itemize}

\item $L(P)\,\subset \, P\,\subset \, G$ is a Levi
subgroup of some parabolic subgroup $P$ of $G$, and

\item $E_{L(P)} \,\subset\, E_G\vert_{U}$ is a holomorphic
reduction of structure group to $L(P)\,\subset\, G$, over the large
open subset $U$ over which $E$ is locally free, with the property
that the section $\varphi\vert_U$ lies in the image of the
natural inclusion
\[
H^0(U,\,\Omega^1_X\otimes{\rm ad}(E_{L(P)}))\,\hookrightarrow
\, H^0(U, \, \Omega^1_X\otimes{\rm ad}(E_G\vert_U))
\]
\end{itemize}
such that the following two hold:
\begin{enumerate}
\item the principal Higgs $L(P)$--bundle
$(E_{L(P)}\, , E' \, ,\psi\, ,
\varphi\, ,\widehat{\varphi})$ is stable,
where $E'$ is the coherent analytic subsheaf of
$E$ generated by ${\rm ad}(E_{L(P)})$ using $\psi$,
and

\item for each
character $\chi$ of $L(P)$ which is trivial on
the center of $G$, the line bundle $E_{L(P)}(\chi)$
over $U$ associated to $E_{L(P)}$ for the character
$\chi$ is of degree zero.
\end{enumerate}
Note that there is a natural
inclusion of ${\rm ad}(E_{L(P)})$ in ${\rm ad}(E_{G})$;
hence there is a natural homomorphism from ${\rm ad}(E_{L(P)})$
to $E_G({\mathfrak g}')$. (See \cite{RS}, \cite{AB}
for the definition of polystable principal bundles.)

We will now define hermitian Yang--Mills--Higgs connections
on principal $G$--sheaves.

Fix a maximal compact subgroup
\begin{equation}\label{eq2a}
K(G)\,\subset\, G\, .
\end{equation}
If $E'_G$ is a holomorphic principal $G$--bundle
over a complex manifold, and $E'_{K(G)} \,\subset\,
E'_G$ is a $C^\infty$ reduction of structure group
of $E'_G$ to $K(G)$, then the $G$--bundle
$E'_{G}$ has a unique complex connection
which is induced by a connection on
$E'_{K(G)}$. This unique connection
will be called the \textit{Chern connection}.

Set
\begin{equation}\label{eq2}
Z \, :=\, G/[G\, , G]
\end{equation}
to be the quotient group, which is a product of
copies of ${\mathbb G}_m \, =\, {\mathbb C}^*$. Note
that $Z$ is a finite quotient of the connected component,
containing the identity element, of the center of $G$.

Let $E_G$ be a holomorphic principal $G$--bundle over a large open
subset $U$ of $X$. Let $\varphi$ be a Higgs field on $E_G$ over $U$.
Let
\begin{equation}\label{eq3}
E_Z\, := \, E_G(Z)
\end{equation}
be the principal $Z$--bundle over $U$ obtained
by extending the structure group of $E_G$ using
the quotient map $G\,\longrightarrow \, Z$
in \eqref{eq2}. The Higgs field on $E_Z$ over $U$
induced by $\varphi$ will be denoted by
$\varphi^z$.

The above defined principal Higgs $Z$--bundle
$(E_Z\, ,\varphi^z)$ extends
to a holomorphic principal Higgs $Z$--bundle over $X$.
Indeed, this follows from the facts that
$Z$ is a product of copies of ${\mathbb C}^*$,
and any holomorphic line bundle over $U$ extends
to a holomorphic line bundle over $X$. To
see that any holomorphic line bundle $L$ over $U$
extends to $X$, consider the determinant line bundle
$\det (\iota_* L)$ over $X$, where $\iota$ is the
inclusion map of $U$ in $X$. See \cite[Ch. V, \S 6]{Ko}
for the construction of the determinant line bundle
of a torsionfree coherent analytic sheaf on $X$;
note that from the condition that the codimension
of the complex analytic set $X\setminus U
\,\subset \, X$ is at least two
it follows that the direct image $\iota_* L$ is
a coherent analytic sheaf on $X$. The holomorphic
extension of $E_Z$ to $X$ is clearly unique.
Since $\varphi^z$ is a holomorphic section of
$\Omega^1_X\otimes_{\mathbb C} \mathfrak{z}$ over $U$,
where $\mathfrak z$ is the Lie algebra of $Z$,
and the codimension of the complex analytic set $X
\setminus \,\subset \, X$ is at least two, the section
$\varphi^z$ extends to a holomorphic section of
$\Omega^1_X\otimes_{\mathbb C} \mathfrak{z}$ over $X$.

Since any Higgs line bundle over $X$ has a
unique hermitian Yang--Mills--Higgs connection,
any holomorphic principal Higgs $Z$--bundle over $X$ also
has a unique hermitian Yang--Mills--Higgs connection.

Let $(E_G\, , E \, ,\psi\, ,\varphi\, ,\widehat{\varphi})$ be a
principal Higgs $G$--sheaf on $(X,\omega)$.
Let $U\,\subset \, X$ be the large
open subset over which $E_G$ is a holomorphic principal $G$--bundle
(see Remark \ref{rem1}). A \textit{hermitian Yang--Mills--Higgs
connection} on $(E_G\, , E \, ,\psi\, ,\varphi\,
,\widehat{\varphi})$ is a Chern connection $\nabla$ on the principal
$G$--bundle $E_G$ over $U$ satisfying the following two conditions:
\begin{enumerate}
\item{} the connection on the principal Higgs $Z$--bundle
$(E_Z\, ,\varphi^z)$ (defined in \eqref{eq3})
induced by $\nabla$ coincides with the unique
hermitian Yang--Mills--Higgs connection on the extension
of $(E_Z\, ,\varphi^z)$ to $X$ (recall that $(E_Z\, ,\varphi^z)$
extends holomorphically
to $M$, and the extension has a unique hermitian Yang--Mills--Higgs
connection); and

\item{} the connection on $E\vert_U$ induced
by $\nabla$ and $\psi$ is an admissible hermitian
Yang--Mills--Higgs connection on the reflexive Higgs sheaf
$(E^{\vee\vee}\, ,\widehat{\varphi})$ (the connection $\nabla$
induces a connection on the associated
vector bundle $E_G({\mathfrak g}')$ in \eqref{00},
and using the isomorphism $\psi$ in \eqref{00}, this induced
connection gives a connection on $E\vert_U$).
\end{enumerate}

Let $(\cE\, ,\varphi)$ and $({\cE}',\varphi')$
be two Higgs sheaves on $X$. Define
\begin{equation}\label{de.cV}
{\mathcal V}\, :=\, (\cE\otimes {\cE}')/\text{torsion}\, .
\end{equation}
The Higgs fields $\varphi$ and $\varphi'$ together induce
a Higgs field $\theta$ on ${\mathcal V}$. The description
of $\theta$ is the following:
\begin{equation}\label{de.th}
\theta\, =\, \varphi\otimes\text{Id}_{{\cE}'} +
\text{Id}_{\cE}\otimes \varphi'\, .
\end{equation}

\begin{lemma}\label{lem1.G}
Assume that the two Higgs sheaves
$(\cE\, ,\varphi)$ and $({\cE}',\varphi')$
are both polystable. Then the
Higgs sheaf $({\mathcal V}\, ,\theta)$,
defined in \eqref{de.cV} and \eqref{de.th}, is also
polystable.
\end{lemma}

\begin{proof}
Consider the double dual $\cE^{\vee\vee}$ equipped with the
Higgs field induced by $\varphi$. This induced Higgs field on
$\cE^{\vee\vee}$ will be denoted by $\varphi^{\vee\vee}$.
Since $(\cE\, ,\varphi)$ is polystable, it follows that
$(\cE^{\vee\vee}\, ,\varphi^{\vee\vee})$ is also polystable
(see Lemma \ref{tors}). Let
$\nabla$ be the unique hermitian Yang--Mills--Higgs connection on
$(\cE^{\vee\vee}\, ,\varphi^{\vee\vee})$ given by
Corollary \ref{YM-cor.}. Similarly, let $\nabla'$
be the unique hermitian Yang--Mills--Higgs connection on
$((\cE')^{\vee\vee}\, ,(\varphi')^{\vee\vee})$; as before,
the Higgs field
on $(\cE')^{\vee\vee}$ induced by $\varphi'$ is denoted
by $(\varphi')^{\vee\vee}$.

Now it is easy to see that the connection
\[
\nabla^{\mathcal V} \, :=\,
\nabla\otimes\text{Id}_{({\cE}')^{\vee\vee}} +
\text{Id}_{\cE}\otimes \nabla'
\]
induces a hermitian Yang--Mills--Higgs connection on
$({\mathcal V}\, ,\theta)$. Consequently, the Higgs sheaf
$({\mathcal V}\, ,\theta)$ is polystable (Corollary
\ref{YM-cor.}). This completes the proof of the lemma.
\end{proof}

\begin{proposition}\label{prop1.G}
Assume that $(\cE\, ,\varphi)$ and $({\cE}',\varphi')$
are both semistable. Then the Higgs
sheaf $({\mathcal V}\, ,\theta)$,
defined in \eqref{de.cV} and \eqref{de.th}, is also
semistable.
\end{proposition}

\begin{proof}
Since $(\cE\, ,\varphi)$ is semistable, there is filtration
of coherent subsheaves
\begin{equation}\label{G.filt}
0\, =\, E_0\, \subsetneq\, E_1\, \subsetneq\, \cdots\,
\subsetneq\, E_{\ell-1} \, \subsetneq\, E_\ell \, =\, \cE
\end{equation}
such that the following hold:
\begin{itemize}
\item $E_i/E_{i-1}$ is torsionfree with
\[
\frac{\deg (E_i/E_{i-1})}{\rk (E_i/E_{i-1})} \,
=\, \frac{\deg \cE}{\rk \cE}
\]
for all $i\in [1\, ,\ell]$,

\item $\varphi (E_i) \, \subset\, \Omega^1_X\otimes E_i$
for all $i\in [0\, ,\ell]$, and

\item for all $i\in [1\, ,\ell]$, the quotient $E_i/E_{i-1}$
equipped with the Higgs field induced by $\varphi$ is
polystable.
\end{itemize}

Let
\[
0\, =\, E'_0\, \subsetneq\, E'_1\, \subsetneq\, \cdots\,
\subsetneq\, E'_{\ell'-1} \, \subsetneq\, E'_{\ell'} \, =\, {\cE}'
\]
be the filtration constructed as in \eqref{G.filt} for the
semistable Higgs sheaf $({\cE}',\varphi')$.
Therefore, the Higgs fields $\varphi$ and $\varphi'$
together induce a Higgs field on
\[
{\mathcal V}_{i,j} \, :=\, ((E_i/E_{i-1})\otimes
(E'_j/E'_{j-1}))/\text{torsion}
\]
for all $i\in [1\, ,\ell]$
and $j\in [1\, ,\ell']$. This Higgs field on
${\mathcal V}_{i,j}$ will be denoted by $\theta_{i,j}$. From
Lemma \ref{lem1.G} it follows that the Higgs sheaf
$({\mathcal V}_{i,j}\, ,\theta_{i,j})$ is polystable for all
$i,j$. Since $\deg (E_i/E_{i-1})/\rk (E_i/E_{i-1})
= \deg \cE/\rk \cE$ and $\deg (E'_j/E'_{j-1})/\rk
(E'_j/E'_{j-1}) = \deg {\cE}'/\rk {\cE}'$, we conclude that
\begin{equation}\label{eq3a}
\frac{\deg {\mathcal V}_{i,j}}{\rk {\mathcal V}_{i,j}} \,
=\, \frac{\deg {\mathcal V}}{\rk {\mathcal V}}\, .
\end{equation}

Let $U\, \subset\, X$ be the dense open subset over which
all $E_i/E_{i-1}$, $i\in [1\, ,\ell]$, and all
$E'_j/E'_{j-1}$, $j\in [1\, ,\ell']$, are locally free. The
complement $X\setminus U$ is a
complex analytic subset of complex codimension
at least two (recall that all $E_i/E_{i-1}$ and
$E'_j/E'_{j-1}$ are torsionfree).

We note that over $U$, the Higgs sheaf
$({\mathcal V}\, ,\theta)$ admits a filtration such that each
successive quotient is a Higgs sheaf of the form
$({\mathcal V}_{i,j}\, ,\theta_{i,j})$ for some $i,j$. We already
noted that each $({\mathcal V}_{i,j}\, ,\theta_{i,j})$ is polystable
satisfying \eqref{eq3a}. Consequently, the Higgs sheaf
$({\mathcal V}\, ,\theta)$ is semistable. This completes the
proof of the proposition.
\end{proof}

\begin{theorem}\label{thm1}
A principal Higgs $G$--sheaf $(E_G\, , E \, ,\psi\,,
\varphi\, ,\widehat{\varphi})$ over a
compact connected K\"ahler manifold
$X$ admits an admissible hermitian Yang--Mills--Higgs connection
if and only if $(E_G\, , E \, ,\psi\,,
\varphi\, ,\widehat{\varphi})$ is polystable. Furthermore,
a polystable principal Higgs $G$--sheaf admits a unique
hermitian Yang--Mills--Higgs connection.
\end{theorem}

\begin{proof}
Since we have proved Proposition \ref{prop1.G}, the proof of
Theorem 4.10 in \cite[p. 227]{AB} goes through to give a
proof of the above theorem. See the final paragraph in
\cite[p. 227]{AB} explaining the issue. We give some details
of the arguments.

Let $U$ be the large open subset of $X$ over which $E_G$ is
holomorphic principal $G$--bundle (see Remark \ref{rem1}). Consider
the reflexive sheaf $\iota_* \text{ad}(E_G\vert_U)$ over $X$, where
$\iota\, :\, U\, \hookrightarrow\, X$ is the inclusion map.
We note that using $\psi$, the direct image
$\iota_* \text{ad}(E_G\vert_U)$ is identified with the direct
sum $E^{\vee\vee}\oplus (X\times \z({\mathfrak g}))$, where
$X\times \z({\mathfrak g})$ is the trivial holomorphic
vector bundle over $X$ with fiber $\z({\mathfrak g})$. For
notational convenience, the sheaf $\iota_* \text{ad}(E_G\vert_U)$
will be denoted by $\text{ad}(E_G)$. The Higgs field $\varphi$
clearly defines a Higgs field on the reflexive sheaf
$\text{ad}(E_G)$; this induced Higgs field on $\text{ad}(E_G)$ will
be denoted by $\varphi'$.

If $\nabla$ is an admissible hermitian Yang--Mills--Higgs connection
on $(E_G\, , E \, ,\psi\,, \varphi\, ,\widehat{\varphi})$, then it
is straight--forward to check that the connection on $\text{ad}(E_G)$
induced by $\nabla$ is an admissible hermitian Yang--Mills--Higgs
connection for the Higgs sheaf $(\text{ad}(E_G)\, , \varphi')$.
Now from Corollary \ref{YM-cor.} it follows that
$(\text{ad}(E_G)\, , \varphi')$ is polystable. From this it is
easy to deduce that $(E_G\, , E \, ,\psi\,, \varphi\, ,
\widehat{\varphi})$ is polystable.

To prove the converse,
assume that $(E_G\, , E \, ,\psi\,, \varphi\, ,
\widehat{\varphi})$ is polystable. Using this assumption
it can be shown that the above defined
Higgs sheaf $(\text{ad}(E_G)\, , \varphi')$
is polystable; the details are in \cite{AB}. Let
$\nabla'$ be the admissible hermitian Yang--Mills--Higgs
connection for the Higgs sheaf $(\text{ad}(E_G)\, , \varphi')$
given by Corollary \ref{YM-cor.}. This connection
$\nabla'$ on $\text{ad}(E_G)$ induces a connection on
$E_G\vert_U$, where $U\, \subset\, X$ is the open subset over
which $E_G$ is a holomorphic principal $G$--bundle. It can
be shown that this induced connection on $E_G\vert_U$ is
an admissible hermitian Yang--Mills--Higgs connection
on $(E_G\, , E \, ,\psi\,, \varphi\, ,\widehat{\varphi})$.
\end{proof}

We have the following analog of the Bogomolov inequality.

Let $(E_G\, , E \, ,\psi\,, \varphi\, ,\widehat{\varphi})$ be a
polystable principal Higgs $G$--sheaf over a compact connected
K\"ahler manifold $X$ equipped with the K\"ahler form $\omega$. Let
$U$ be the large open subset of $X$ over which $E_G$ is holomorphic
principal $G$--bundle (see Remark \ref{rem1}). Consider the reflexive
sheaf ${\rm ad}(E_G)\, :=\, \iota_* {\rm ad}(E_G\vert_U)$ over $X$,
where $\iota\, :\, U\, \hookrightarrow\, X$ is the inclusion map.
Let $\varphi'$ be the Higgs field on ${\rm ad}(E_G)$ induced by
$\varphi$. Let $\nabla'$ be the (singular) connection on ${\rm
ad}(E_G)$ induced by the admissible hermitian Yang--Mills--Higgs
connection on $(E_G\, , E \, ,\psi\,, \varphi\,
,\widehat{\varphi})$. (It was noted in the proof of Theorem
\ref{thm1} that $\nabla'$ coincides with the admissible hermitian
Yang--Mills--Higgs connection on the polystable Higgs sheaf $({\rm
ad}(E_G) \, , \varphi')$ given by Corollary \ref{YM-cor.}.)

\begin{proposition}\label{pr.bo.}
With the above notation,
$$
(2\dim_{\mathbb C}{\mathfrak g}{\cdot}(c_2({\rm ad}(E_G)))-
(\dim_{\mathbb C} {\mathfrak g} -1)c_1(
{\rm ad}(E_G))^2)\omega^{d-2}\, \geq\, 0\, ,
$$
where ${\mathfrak g}$ is the Lie algebra of $G$. Furthermore,
the equality holds if and only if $U\, =\, X$ and the
connection $\nabla'+ \varphi'+  (\varphi')^*$ on
${\rm ad}(E_G)$ is projectively flat.
\end{proposition}

\begin{proof}
Since the connection $\nabla'$ is the admissible hermitian
Yang--Mills--Higgs connection on the polystable Higgs sheaf
$({\rm ad}(E_G) \, , \varphi')$, this proposition follows from
the proof of Proposition 3.4 in \cite[p. 878]{Si}. See
also Corollary 3 in \cite[p. 40]{BS}.
\end{proof}


\end{document}